\documentclass[reqno]{amsart}
\usepackage{color}
\usepackage{graphicx}

\def\C{\mathcal {C}}
\def\CC{\mathbb{C}}

\def\N{\mathbb{ N}}

\def\O{\mathcal{ O}}

\def\i{\mathrm{i}}
\def\e{\mathrm{e}}
\def\d{\mathrm{d}}
\newtheorem{theorem}{Theorem}

\newtheorem{proposition}[theorem]{Proposition}
\theoremstyle{remark}
\newtheorem{remark}[theorem]{Remark}

\begin{document}
\title[Numerical integration of functions of a rapidly rotating
phase]%
{Numerical integration of functions \\of a rapidly rotating phase}
\date{\today}

\author{Haidar Mohamad}
\author{Marcel Oliver}
\address{School of Engineering and Science \\
 Jacobs University \\
 28759 Bremen \\
 Germany}


\begin{abstract}
We present an algorithm for the efficient numerical evaluation of
integrals of the form
\[
  I(\omega) = \int_0^1 F( x,\e^{\i \omega x}; \omega) \, \d x
\]
for sufficiently smooth but otherwise arbitrary $F$ and
$\omega \gg 1$.  The method is entirely ``black-box'', i.e., does not
require the explicit computation of moment integrals or other
pre-computations involving $F$.  Its performance is uniform in the
frequency $\omega$.  We prove that the method converges exponentially
with respect to its order when $F$ is analytic and give a numerical
demonstration of its error characteristics.
\end{abstract}
\maketitle

\section{Introduction}

We consider the problem of numerical approximation of integrals of the
form 
\begin{equation}
  \label{MainProb}
  I(\omega) = \int_0^1 F( x,\e^{\i \omega x}; \omega) \, \d x \,, 
\end{equation}
where $F \colon [0,1] \times \mathbb U \to \CC$, $\mathbb U$ denotes
the unit circle in the complex plane, and $\omega>0$.  $F$ may, in
addition, depend parametrically on $\omega$.  In most of the
following, we will not write out this parametric dependence explicitly
except where it matters for a precise statement of the quadrature
error estimate.
Classical quadrature formulas require that the number
of integration nodes grows linearly in the frequency $\omega$, so that
the problem becomes increasingly intractable when the frequency is
large.

One of the earliest integration methods for integrals of this type is
due to Filon \cite{Filon}, who studied the special case
\begin{equation}
  F( x,\e^{\i \omega x}; \omega)
  = F( x,\e^{\i \omega x})
  = f(x) \, \e^{\i \omega x} \,.
\end{equation}
Filon replaced the function $f$ by a polynomial approximation so that
the resulting moment integrals could be computed analytically.  The
method has been refined and extended by many authors
\cite{Flinn,Shampine:2013:EfficientFM,VVoorenAndLinde}.  Other methods
use interpolatory formulas and formulas which are based on the
integration between the zeros of $\cos(\omega x)$ and $\sin(\omega x)$
\cite{Longman,Miklosko,PiessensP:1971:NumericalMI}.

Most subsequent work went into oscillatory integrals of the form
\begin{equation}\label{HOIg}
  I(\omega) = \int_0^1 f(x) \, \e^{\i \omega g(x)} \, \d x 
\end{equation}
which is a more subtle problem when the phase function $g$ has
stationary points.  
Levin \cite{Levin} suggested to convert the integrand into a perfect
derivative.  He seeks a function $p$ satisfying
\begin{equation}
  \frac{\d}{\d x} \bigl( p(x) \, \e^{\i \omega g(x)} \bigr)
  = f(x) \, \e^{\i \omega g(x)} \,,
\end{equation}
a differential equation which can be solved by collocation.  The value
for the integral is then recovered via
\begin{equation}
  I(\omega)
  = \int_0^1 f(x) \, \e^{\i \omega g(x)} \, \d x
  = p(0) \, \e^{\i \omega g(0)} - p(1) \, \e^{\i \omega g(1)} \,.
\end{equation}

Olver \cite{Olver:2007:MomentFN} suggest a choice of approximation
basis for $f$ which is compatible with integration against $\e^{\i
\omega g(x)}$ so that Filon-type ideas can be extended to problem
\eqref{HOIg}.

A third approach is based on asymptotic expansion in inverse powers of
the frequency.  Noting that
\begin{equation}\label{IserlesNorsett}
  I(\omega)
  \sim \sum_{k = 0}^{p-1} \frac{1}{(-\i \omega)^{k+1}} \,
    \biggl(
      \frac{\e^{\i \omega g(1)}}{g'(1)} \, f_m(1)
      - \frac{\e^{\i \omega g(0)}}{g'(0)} \, f_m(0)
    \biggr) \,,
\end{equation}
where 
\begin{equation}
  f_0(x) = f(x)
  \quad \text{and} \quad
  f_{m+1}(x) = \frac{\d}{\d x} \frac{f_m(x)}{g'(x)} \,,
\end{equation}
one can show that the error in \eqref{IserlesNorsett} is
$\O(\omega^{-p-1})$, so the method is accurate so long as $\omega$ is
large.  Iserles and N{\o}rsett
\cite{IserlesAndNorsett,IserlesAndNorsett2} modify
\eqref{IserlesNorsett} as to not require the computation of
derivatives at the endpoint while producing errors comparable to other
asymptotic and Filon-type methods.  For reviews of available methods
and further references, see
\cite{DeanoHI:2018:ComputingHO,HuybrechsO:2009:HighlyOQ}.

None of the methods mentioned so far, however, extends to
\eqref{MainProb} in the general case, i.e., without exploiting a
particular form of the function $F$.  We encountered integrals of this
form when extending uniformly accurate exponential integrators for the
Klein--Gordon equation in the non-relativistic limit, first suggested
by Baumstark \emph{et al.}\ \cite{BaumstarkFS:2018:UniformlyAE}, to
problems with more general nonlinearities
\cite{MohamadO:2019:UniformlyAT}.

In this paper, we derive a uniformly accurate quadrature scheme that
is completely ``black-box'', i.e., can be applied to any function $F$
without $F$-specific pre-computations.  It is based on Gauss
quadrature for sums detailed in Section~\ref{s.summation} below.  We
show that the quadrature error is exponentially small in $n$ when $F$
is analytic.

To motivate our approach, let  $T = 2\pi /\omega$ denote the period of
$x\mapsto \e^{\i \omega x}$.  Then there exist $N \in \N$ and  $\alpha
\in [0, 1)$ such that
\begin{equation}
  (N + \alpha) \, T = 1 \,.
\end{equation}
Let now $x_j$ be the $N+1$ equidistant points
\begin{equation}
  x_j = -1 + \frac{2 j}{N-1}, \qquad 0 \leq j \leq N  \,.
\end{equation}
We can write
\begin{align}
  I(\omega)
  & = \sum_{j=0}^{N-1} \int_{j T}^{(j+1)T} F(x,\e^{\i \omega x}) \, \d x
      + \int_{NT}^{(N+\alpha)T} F(x,\e^{\i \omega x}) \, \d x
      \notag\\
  & = T \sum_{j=0}^{N-1} \int_0^1 F(T(t + j), \e^{2 \pi \i t}) \, \d t
      + T \int_0^\alpha F(T(t + N), \e^{2 \pi \i t}) \, \d t
      \notag \\
  & = T \sum_{j=0}^{N-1} I_1(x_j) + T \, I_\alpha (x_N) 
  \label{IntegTrans}
\end{align}
with
\begin{equation}
  I_b(y) = \int_0^b F \bigl( Tt + \tfrac12 \, T \, (N-1) (y+1),
    \e^{2 \pi \i t} \bigr) \, \d t \,.
  \label{e.inner}
\end{equation}
When $F$ depends parametrically on $\omega$, $I_b$ inherits this
parametric dependence.  Importantly, \eqref{e.inner} shows that
$I_b(y)$ is otherwise independent of $\omega$ so that, for fixed $y$,
each $I_b(y)$ can be evaluated easily via any traditional quadrature
rule; errors are uniform in $\omega$ as all derivatives of the
integrand are uniform in $\omega$.  Moreover, $I_1(y)$ varies slowly
as a function of $y$.  Thus, the sum on the right hand side of
\eqref{IntegTrans} could be seen as a Riemann sum,
\begin{equation}
  2T \sum_{j=0}^{N-1} I_1(x_j)
  =  \int_{-1}^1 I_1(y) \, \d y + O(\omega^{-1}) \,,
  \label{e.riemannsum}
\end{equation}
where the right hand integral could, again, be approximated by any
traditional quadrature rule.  Since $NT < 2 \pi$, this approximation
is uniform in $\omega$.

The resulting method would be efficient and has an error that is
asymptotically small for large $\omega$.  However, it turns out that
we can do even better, by-passing the Riemann sum approximation
\eqref{e.riemannsum} with its $O(\omega^{-1})$-error entirely: Sums
with a slowly varying summand can be evaluated effectively via Gauss
quadrature for sums with a small number of evaluations, just like
Gauss quadrature for integrals.  Gauss quadrature for sums has been
described by Area \emph{et al.}\
\cite{AreaDG:2014:ApproximateCS,AreaDG:2016:ApproximateCS} but, to the
best of our knowledge, has never been applied in the context of
oscillatory integrals.

The remainder of the paper is structured as follows.  Gauss quadrature
for sums is detailed in Section~\ref{s.summation}, leading to a
complete statement of the algorithm.  Section~\ref{s.convergence}
gives a simple estimate for the quadrature error.  Finally, in
Section~\ref{s.numerics}, we demonstrate that the method is easy to
implement and performs well.

\section{Gauss quadrature for sums}
\label{s.summation}

Let $N$ be a positive integer, arbitrary but fixed in the following.
Then there exists a unique quadrature formula
\begin{equation}
  \label{GramGaussQuad}
  S(G) \equiv \frac{2}{N} \sum_{j=0}^{N-1}G(x_j)
  \approx \sum_{k=1}^n w_{k} \, G(s_{k})
  \equiv S_n(G) \,,  
\end{equation}
which is exact for all polynomials of degree $ \leq 2n-1$.

The construction uses so-called Gram polynomials $p_m$,
$m = 0, \dots, N-1$, which are defined, up to choice of sign, by their
orthonormality with respect to a discrete equidistant sum, namely
\begin{equation}
  \sum_{j=0}^{N-1} p_l(x_j) \, p_m(x_j) = \delta_{lm} \,.
\end{equation}
For fixed $n<N$, the quadrature nodes $\{s_k\}$ are the zeros of the
Gram polynomial of degree $n$.  Then
\begin{equation}
  q_k(x) = \frac{p_n(x)}{x-s_k} - \frac{a_n}{a_{n-1}} \, p_{n-1}(x)
\end{equation}
is a polynomial  of degree $n-2$, where $a_m$ denotes the leading
coefficient of $p_m$.

For any polynomial $p$ of degree $\leq 2n-1$ that vanishes at all the
nodes $s_l$ except for $s_k$, \eqref{GramGaussQuad} implies that
\begin{equation}
  w_k = \frac{2}{N p(s_k)}\sum_{j=0}^{N-1}p(x_j) \,.
\end{equation}
Taking
\begin{equation}
  p(x) = \frac{p_n(x) \, p_{n-1}(x)}{x - s_k} \,,
\end{equation}
in particular, we obtain
\begin{equation}
  w_k = \frac{2}{N \, p'_n(s_k) \, p_{n-1}(s_k)}
    \sum_{j=0}^{N-1}\frac{p_n(x_j) \,p_{n-1}(x_j)}{x_j - s_k} \,.
\end{equation}
Since $q_k$ is of degree $n-2$, it is orthogonal to $p_{n-1}$.  We
conclude that
\begin{equation}
  w_k = \frac{a_n}{a_{n-1}} \, \frac{2}{N\,p'_n(s_k) \, p_{n-1}(s_k)} \,.
\end{equation}

The Gram polynomials $p_n$ can be expressed in closed form in terms of
the hypergeometric function ${}_3F_2$ by
\begin{equation}
  p_n(x) = (-1)^n \, \sqrt{\frac{(2n +1) \, (N-n )_n}{(N)_{n+1}}} \,
  {}_3F_2 \biggl( \begin{matrix} -n,& n+1, & (1-N)(1+x)/2\\  & 1,&
  1-N \end{matrix} \,\bigg\vert\,  1 \biggr) \,,
\end{equation}
\cite[Equations 7.13.7 and 7.13.15]{Hildebrand}, 
with Pochhammer symbol defined by
\begin{equation}
  (A)_0 = 1 \,, \quad
  (A)_n = A (A+1)(A+2)\cdots (A+n-1) \text{ for } n \in \N^\ast \,.
\end{equation}
By expanding the finite series representation of ${}_3F_2$, we find
that the leading order coefficient is given by
\begin{equation}
  a_n = \sqrt{\frac{(2n+1)(N-n-1)!}{(N+n)!}} \,
    \frac{(2n)! \, (N-1)^n}{2^n \, (n!)^2}
\end{equation}
so that
\begin{equation}
  \frac{a_n}{a_{n-1}} = \frac{N-1}{n} \,
    \sqrt{\frac{4n^2 -1}{N^2 - n^2}} \,.
\end{equation}
For details, see \cite[p.~348]{Hildebrand} and
\cite[p.~170]{OrtPolDisVar}.  We note that the expressions in
\cite{AreaDG:2014:ApproximateCS,AreaDG:2016:ApproximateCS} differ from
the ones given here due to the different choice of nodes in the
definition of the discrete inner product \eqref{e.inner}.

Applying the Gauss summation formula \eqref{GramGaussQuad} to
\eqref{IntegTrans}, we obtain the final quadrature approximation
\begin{equation}
  I_\text{comp}(\omega;n)
  = \frac{NT}2 \sum_{k=1}^{n} w_k \, I_1(s_k)
    + T \, I_\alpha(x_N) \,.
  \label{e.In}
\end{equation}

\section{Convergence analysis}
\label{s.convergence}

In the following, we use the Chebyshev approximation to quantify the
error of the Gauss quadrature formula for sums.  To fix notation, let
$G$ be a continuous function on $[-1, 1]$.  We write
\begin{equation}
  G_n(x) = \sum_{j = 0}^n a_j \, T_j(x)
\end{equation}
to denote its polynomial approximation of degree $n$ obtained by
truncating the Chebyshev series at order $n$.  Here,
$T_j(x) = \cos(j \arccos (x))$ is the Chebyshev polynomial of degree
$j$ and the coefficients are given by
\begin{subequations}
\begin{gather}
  a_0 = \frac{1}{\pi}\int_{-1}^1 \frac{G(x)}{\sqrt{1-x^2}} \, \d x \,, \\
  a_j = \frac{2}{\pi}\int_{-1}^1
    \frac{G(x) \, T_j(x)}{\sqrt{1-x^2}} \, \d x \text{ for } j \geq 1 \,.
\end{gather}
\end{subequations}
We write $\lVert \, \cdot \, \rVert$ to denote the supremum norm on
$[-1,1]$ and define
\begin{gather}
  d_n = \lVert G - G_n \rVert \,.
\end{gather}

\begin{proposition}\label{PropGsGmEr}
Let $G \in \C([-1,1])$ and $S$ and $S_n$ be defined as in
\eqref{GramGaussQuad}.  Then
\begin{equation}\label{GsGmEstim1}
  \lvert S(G) - S_n(G) \rvert \leq 4 \, d_{2n-1} \,.
\end{equation}
\end{proposition}

\begin{proof}
As \eqref{GramGaussQuad} is exact for polynomials of degree
$\leq 2n-1$, we have $(S-S_n)(G) = (S-S_n)(G-G_{2n-1})$.  Hence,
\begin{align}
  \lvert S(G) - S_n(G) \rvert
  & \leq \lvert S(G-G_{2n-1}) \rvert + \lvert  S_n(G-G_{2n-1}) \rvert
    \notag \\
  & \leq 2 \, d_{2n-1}
    + \sum_{k=1}^n \, \lvert w_{k} \rvert \,  d_{2n-1} \,.
\end{align}
Since the weights are non-negative \cite{OrtPolDisVar} and formula
\eqref{GramGaussQuad} is interpolatory,
\begin{equation}
  \sum_{k=1}^n \, \lvert w_{k} \rvert = 2
\end{equation}
which implies \eqref{GsGmEstim1}.
\end{proof}

When $G$ is smooth, the error of the Chebyshev approximation
satisfies the following strong bounds.
\begin{theorem}[{\cite[Theorem 4.3]{Trefethen:2008:GaussQB}}]\label{ThmdnEstim}
Let $G \in \C([-1,1])$ be such that $G, G', \dots, G^{(m-1)}$ are
absolutely continuous and
\begin{equation}
  \left\| \frac{G^{(m)}}{\sqrt{1-x^2}} \right\|_1
  \equiv V < \infty
\end{equation}
for some $m\geq 1$.  Then, for every $n \geq m+1$,
\begin{equation}
  d_n \leq \frac{2V}{\pi \, m \, (n-m)^m} \,.
\end{equation}
Moreover, if $G$ is analytic with $|G(z)| \leq M$ in the region
bounded by the ellipse with foci $\pm 1$ and major and minor semiaxis
lengths summing to $\rho > 1$, then for every $n \geq 0$,
\begin{equation}
  d_n \leq \frac{2 M}{(\rho-1) \, \rho^n} \,.
\end{equation}
\end{theorem}

Applying Proposition~\ref{PropGsGmEr} and Theorem~\ref{ThmdnEstim} to
the function $G(y)=I_1(y;\omega)$ directly yields the following error
estimate for the oscillatory quadrature.

\begin{theorem}\label{mainThm}
Fix $\omega_0\geq 4 \pi$ and $m \in \N$.  Let
$F \colon [0,1] \times \mathbb U \to \CC$ be continuous.  Assume
further that the $m-1$ first derivatives of $I_1(y;\omega)$ defined in
\eqref{e.inner} are absolutely continuous on $[-1, 1]$ and that there
exists a constant $V$ such that
\begin{equation}
  \left\| \frac{I_1^{(m)}(\,\cdot\,;\omega)}{\sqrt{1-y^2}} \right\|_1
  \leq V
\end{equation}
uniformly with respect to  $\omega \geq \omega_0$.  Then, for every
$n \geq m/2 +1$,
\begin{equation}
 \biggl \lvert I(\omega) - \frac{NT}2 \sum_{k=1}^n w_k I_1(s_k)
  - I_\alpha(x_N)\biggr \rvert
  \leq \frac{4  V}{m \, (2n-1-m)^m} \,.
  \label{e.algebraic}
\end{equation}
Moreover, if for some $\rho >1$ the function $I_1(y)$ is analytic with
$|I_1(y)| \leq M$ in the region bounded by the ellipse with foci
$\pm 1$ and major and minor semiaxis lengths summing to $\rho > 1$,
uniformly in $\omega \geq \omega_0$, then for every $n \geq 0$,
\begin{equation}
  \biggl\lvert I(\omega) - \frac{NT}2
  \sum_{k=1}^n w_k I_1(s_k) - I_\alpha(x_N) \biggr\rvert  \leq  
  \frac{4 M}{(\rho-1) \, \rho^{2n-1}} \,.
  \label{e.exponential}
\end{equation}
\end{theorem}

\begin{remark}
The assumption $\omega_0 \geq 4 \pi$ ensures that $N \geq 2$ so that
the Gram polynomials are well defined.  When $\omega<4 \pi$,
$x\mapsto F(x, \e^{i\omega x})$ is not highly oscillatory so that
classical methods are applicable.
\end{remark}

\begin{remark} \label{r.sufficient}
It is possible to formulate sufficient conditions which directly refer
to $F$.  Since
\begin{align}
  \left\| \frac{I_1^{(m)}}{\sqrt{1-y^2}} \right\|_1
  & \leq \frac{1}{2^{m+1} \pi } \int_{-1}^1 \int_0^{2\pi} \frac{| \partial^m_x
           F \bigl( \omega^{-1} s+ \tfrac12 \, T \, (N-1) (y+1),
          \e^{\i s} ; \omega \bigr)|}{\sqrt{1-y^2}} \, \d s \, \d y
     \notag \\
  & \leq \frac{\pi }{2^m} \,
         \sup_{x, z} \left| \partial^m_x F(x, z; \omega) \right| ,
    \label{e.sufficient}
\end{align}
estimate \eqref{e.algebraic} holds whenever the first $m$
$x$-derivatives of $F$ are uniformly bounded with respect to $x$, $z$,
and $\omega$.  Likewise, estimate \eqref{e.exponential} holds whenever
$F$ is analytic in its first argument with a radius of analyticity
that is uniform with respect to $x$, $z$, and $\omega$.  However,
Theorem~\ref{mainThm} as stated is stronger because $I_1(y; \omega)$
may be uniformly analytic even if $F$ is not uniformly analytic in its
first argument, as the example given in the next section shows.
Moreover, estimate \eqref{e.sufficient} for $V$ and analogous
estimates for $M$ will generally over-estimate the constants.
\end{remark}

\section{Implementation and numerical test}
\label{s.numerics}
In the discussion above, we have not specified a quadrature rule for
the ``inner integrals'' \eqref{e.inner}.  The quadrature error there
depends on the smoothness of $F$ in both arguments (in fact, more
strongly on the second).  Since the inner quadrature is always over a
full period of sine and cosine functions, the required number of
quadrature points is typically larger, but not excessively larger,
than the number of quadrature nodes for the outer sum.

We consider the example
\begin{equation}
  F(x,\e^{\i \omega x}; \omega, a)
  = \frac{2 x- \omega\sin(\omega x)}%
         {2\sqrt{a+x^2 + \cos(\omega x)}}
\end{equation}
with $a \geq 1$.  Here, standard quadrature libraries fail or perform
increasingly poorly when $\omega$ becomes large.  On the other hand,
the exact value of the integral can be computed directly, it is
\begin{equation}
  I_{\text{exact}}(\omega;a)
  = \sqrt{a+1 +\cos(\omega)} - \sqrt{a +1} \,.
\end{equation}
Moreover, the inner integral \eqref{e.inner} can also be computed
explicitly:
\begin{align}
  & I_1(y; \omega,a)
  = \frac{1}{4 \pi} \int_0^{2 \pi}
    \frac{2 \pi  (N-1)(y+1) +2 s- \omega^2\sin(s)}%
         {\sqrt{a \omega^2 + (\pi  (N-1)(y+1) +  s)^2
          + \omega ^2 \cos(s)}} \, \d s
    \notag \\ 
  & = \frac{2 \pi \, ((N -1)\, y + N)}{
       \sqrt{(a+1) \omega^2 +\pi^2 (N-1)^2 (y+1)^2}
       + \sqrt{(a+1) \omega^2 + \pi^2 ((N-1) y+N+1)^2}} \,.
\end{align}
Since $N=O(1/\omega)$, it is obvious that $I_1$ is uniformly bounded
on its domain of analyticity and estimate \eqref{e.exponential} of
Theorem~\ref{mainThm} applies.

Note, however, that when $a = 1$, the pole asymptotic to
$x=-\pi/\omega$ approaches the interval of integration as
$\omega \to \infty$, so that the sufficient conditions of
Remark~\ref{r.sufficient} are not satisfied.  This is reflected by the
fact that a Gaussian quadrature of fixed order performs poorly on the
inner integral---in a small region near $s=\pi$, corresponding to
$t=1/2$, the inner integrand develops steep gradients as $\omega$
becomes large---but standard adaptive quadrature implementations have
no difficulty dealing with this case and perform well.  In our example
implementation, we use a binding to the well known \textsc{quadpack}
Fortran library \cite{PiessensDU:1983:QUADPACK}.

\begin{figure}
\centering
\includegraphics[width=0.8\textwidth]{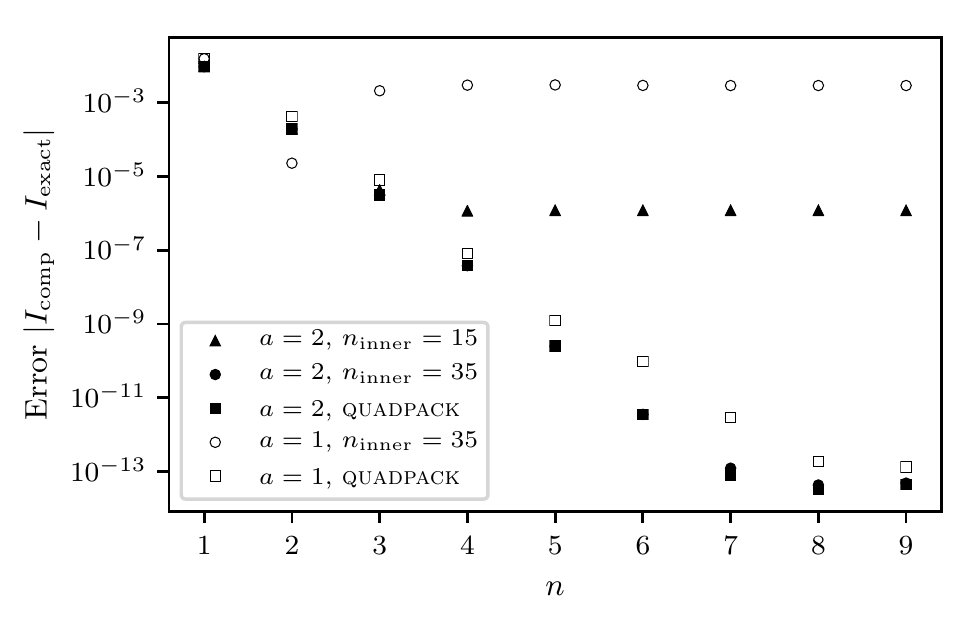}
\caption{Scaling of the error with the number of the outer Gauss
quadrature notes $n$.  We compare different schemes for the inner
quadrature for the case when $F$ is uniformly analytic ($a=2$, filled
marker symbols) with the case when uniform analyticity fails and
non-adaptive inner Gauss quadrature struggles ($a=1$, empty marker
symbols).  In this example, $\omega = 10^4$ is fixed.}
\label{f.1}
\end{figure}

\begin{figure}
\centering
\includegraphics[width=0.8\textwidth]{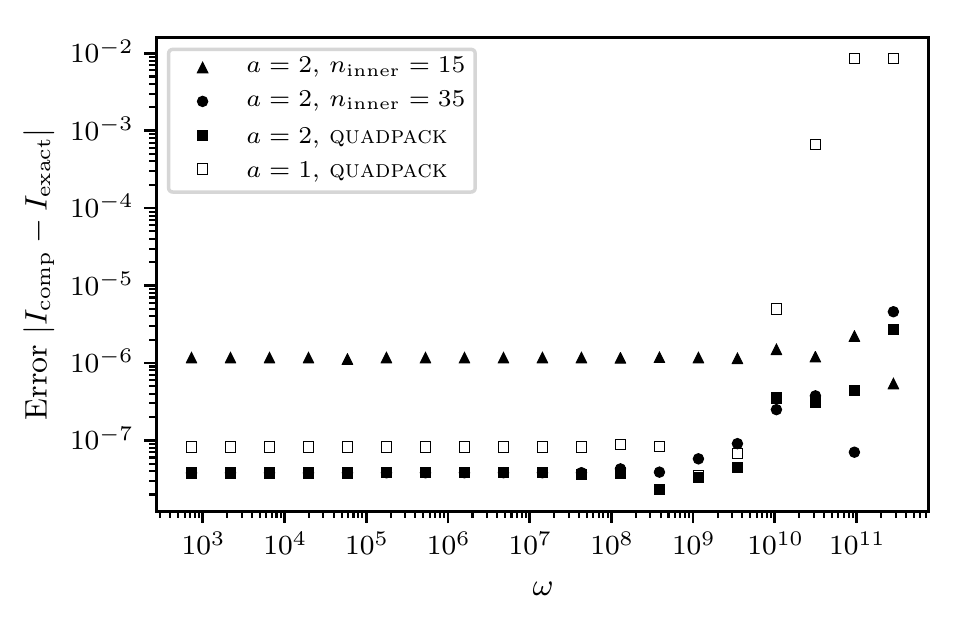}
\caption{Demonstration of the uniformity of the Gauss summation scheme
with respect to the fast frequency $\omega$.  For very large values of
$\omega$, accuracy is necessarily lost due to the loss of significant
digits in the evaluation of the trigonometric functions in
double-precision floating point.  In this example, the number of outer
Gauss quadrature nodes is fixed at $n=4$.}
\label{f.2}
\end{figure}

Figure~\ref{f.1} compares the scaling of the error of our oscillatory
quadrature rule with the order of the Gauss summation $n$ for
different choices of the inner quadrature.  When $a=1$ and the inner
integrand is not uniformly analytic, only an adaptive inner quadrature
performs well.  When $a=2$, uniform analyticity holds and the inner
integral can be calculated effectively by a moderate order classical
Gauss quadrature ($n_{\text{inner}} \geq 35$ gives errors comparable
to errors achievable with \textsc{quadpack}) .  Figure~\ref{f.2}
illustrates the uniformity of the error as a function of $\omega$.

We note that the Gauss summation nodes $s_k$ and weights $w_k$ depend
on $\omega$, so they must be re-computed whenever $\omega$, hence $N$,
is changed.  The Gram polynomials themselves are polynomials of degree
$n$ with coefficients which, up to normalization, are polynomials in
$N$ of degree $n$.  Thus, the polynomial data can be pre-computed and
stored in an integer array of size $n^2$ and evaluated in $O(n^2)$
operations.  The roots are found with the
Weierstrass--Dochev--Durand--Kerner algorithm which is known to
converge rapidly for Gram polynomials
\cite{AreaDG:2016:ApproximateCS}.  Since the classical Gauss
quadrature points---the continuum limit of Gauss summation---provide a
good initial guess, this algorithm reaches excellent accuracy in a
small number of iterations which is uniform in $N$.  Moreover, all
$N$-dependent terms need to be evaluated only once, so that the
overall complexity of the root finding step remains at $O(n^2)$.  In
our example implementation, provided as supplementary material to the
manuscript, we use a symbolic mathematics package for all polynomial
manipulations.  This adds some run-time overhead but leads to a
transparent and still reasonably fast implementation.

The complexity of the overall quadrature formula is the complexity of
the evaluation of the weights, which can be done at $O(n^2)$ as all
$N$-dependent terms need to be evaluated only once, times
$n_{\text{inner}}$, the complexity of the inner quadrature, which is
problem-dependent as discussed above.  If several integrals with the
same frequency $\omega$ are performed, the quadrature weights can be
precomputed and the complexity per evaluation drops to
$O(n \cdot n_{\text{inner}})$.  Also, the required number of function
evaluations is always $n \cdot n_{\text{inner}}$.  Since, in many
cases, order $n=6$ is already very accurate and order $n=10$ is mainly
limited by the floating point error, and provided the inner
integration is sufficiently well-behaved, the method is very effective
in practice.

\section*{Acknowledgments}

The work was supported by German Research Foundation grant OL-155/6-2.
MO further acknowledges support through German Research Foundation
Collaborative Research Center TRR 181 under project number 274762653.

\bibliographystyle{siam}
\bibliography{kg.bib}

\end{document}